\documentclass[12pt, a4]{amsart}
\usepackage{amscd, amsmath, amssymb}
\usepackage{fullpage}

\theoremstyle{plain}
\newtheorem{thm}{Theorem}[section]

\newtheorem{pro}[thm]{Proposition}
\theoremstyle{definition}
\newtheorem{defn}[thm]{Definition}
\newtheorem{ex}[thm]{Example}
\newtheorem{rem}[thm]{Remark}

\numberwithin{equation}{section}

\newcommand{\R}{\mathbb{R}}
\newcommand{\N}{\mathbb{N}}

\begin{document}

\title[Positive solutions for elliptic systems]{Nonzero positive solutions of a multi-parameter elliptic system with functional BCs}  

\date{}

\author[G. Infante]{Gennaro Infante}
\address{Gennaro Infante, Dipartimento di Matematica e Informatica, Universit\`{a} della
Calabria, 87036 Arcavacata di Rende, Cosenza, Italy}%
\email{gennaro.infante@unical.it}%

\begin{abstract} 
We prove, by topological methods, new results on the existence of nonzero positive weak solutions for a class of multi-parameter second order elliptic systems subject to functional boundary conditions. The setting is fairly general and covers the case of multi-point, integral and nonlinear boundary conditions. We also present a non-existence result. We provide some examples to illustrate the applicability our theoretical results.
\end{abstract}



\maketitle

\section{Introduction}
In this paper we discuss the solvability of the multi-parameter system of second order elliptic equations subject to functional boundary conditions
\begin{equation}
  \label{ellbvp-intro}
 \left\{
\begin{array}{ll}
   L_iu_i(x)=\lambda_i f_i(x,u(x)), &  x\in \Omega,\quad i=1,2,\ldots,n, \\
 B_iu_i(x)=\eta_i h_{i} [u], & x\in \partial \Omega,\quad i=1,2,\ldots,n,
\end{array}
\right.
\end{equation}
where  $\Omega\subset \R^m$ ($m\geq 2$) is a bounded domain with sufficiently regular boundary, $L_i$ is a strongly uniformly elliptic operator, $B_i$ is a first order boundary operator, 
 $u=(u_1,\dots, u_n)$, $f_i$ is a continuous function, $h_{i}$ is a suitable compact functional, $\lambda_i, \eta_i$ are parameters.

A motivation for studying this kind of boundary value problems (BVPs) is that they often occur in physical applications. In order to illustrate this fact, take $n=1$, $m=2$  and consider the BVP
\begin{equation}
  \label{ellbvp-intro-ex}
 \left\{
\begin{array}{ll}
-\Delta u(x)= f(x, u(x)), & \|x\|_2<1, \\
u(x)= \eta u(0), & \|x\|_2=1,%
\end{array}
\right.
\end{equation}
where $\|\cdot\|_2$ is the Euclidean norm.
The BVP~\eqref{ellbvp-intro-ex}
can be used as a model for the steady-states of the temperature of a heated disk of radius~$1$, where a controller located in the border of the disk
adds or removes heat in manner proportional to the temperature registered by a sensor located in the center of the disk.  In the context of ODEs, a good reference 
for this kind of thermostat problems is the recent paper~\cite{webb-therm}.

The assumptions we make on the functionals $h_{i}$ that occur in~\eqref{ellbvp-intro} are fairly weak and allow to cover, \emph{for example}, the special cases of \emph{multi-point} boundary conditions (BCs) of the form
\begin{equation}\label{mpointbcs}
h_{i}[u]=\sum_{k=1}^{n} \sum_{j=1}^{N}\hat{\alpha}_{ijk}u_k(\omega_j),
\end{equation}
where $\hat{\alpha}_{ijk}$ are non-negative coefficients and $\omega_j\in \Omega$, or \emph{integral} BCs of the type
\begin{equation}\label{intbcs}
h_{i}[u]=\sum_{k=1}^{n}\int_{\Omega}\hat{\alpha}_{ik}(\omega)u_k(\omega)\,d\omega,
\end{equation}
where $\hat{\alpha}_{ik}$ are non-negative continuous functions on $\overline{\Omega}$. Note that the functionals $h_{i}$ in~\eqref{mpointbcs} and~\eqref{intbcs} allow an interaction between the components of the solution. 

There exists a wide literature on multi-point, integral and, more in general, nonlocal BCs. As far as we know multi-point BCs have been studied firstly by Picone~\cite{Picone} in the context of ODEs. For an introduction to nonlocal BCs, we refer the reader to the
reviews~\cite{Conti, rma, sotiris, Stik, Whyburn} and the papers~\cite{kttmna, ktejde, Pao-Wang, jw-gi-jlms}.

Note that our approach is not restricted to \emph{linear} functionals like~\eqref{mpointbcs} and~\eqref{intbcs}, we may also deal with the case of \emph{nonlinear} BCs. These type of BCs also make physical sense; for example the BVP~\eqref{ellbvp-intro-ex} might be modified in order to take into account a nonlinear response of the controller, by having a nonlinear, nonlocal BC of the form 
\begin{equation}\label{nonlin-nonloc-bc}
u(x)= \hat{h}(u(0)),\quad x\in \partial \Omega,
\end{equation}
where $\hat{h}$ is a continuous function. In the context of radial solutions of PDEs on annular domains, 
conditions similar to~\eqref{nonlin-nonloc-bc} have been investigated recently in~\cite{genupa, go-mi-ro, dunn-wang, Goodrich1, Goodrich2}. We stress that nonlinear BCs have been widely studied for different classes of differential equations, nonlinearities and domains, we refer the reader to~\cite{Amann-rev, Cabada1, genupa, Goodrich3, Mawin-Schm, Pao, rma2, yz-mw} and references therein; in particular, the method of upper and lower solutions has been employed for the System~\eqref{ellbvp-intro} in the case of \emph{non-homogeneus} (not necessarily constant) BCs in~\cite{Amann-rev} and in the case of nonlinear BCs (when $\lambda_i=\eta_i=1$) in~\cite{rma2, Pao}.

We highlight that the existence of positive solutions of the System~\eqref{ellbvp-intro} with \emph{homogeneous} BCs has been recently discussed  in~\cite{lan1,lan2} (in the \emph{sublinear} case) and in~\cite{Cid-Infante} (under monotonicity assumptions on the nonlinearities). Our theory can be applied also in this case, by considering $h_{i}[u]\equiv 0$. We do not assume global restrictions on the growth nor we assume monotonicity of the nonlinearities, thus complementing the results in~\cite{Cid-Infante, lan1, lan2}.

We prove, by means of classical fixed point index, the existence of one nontrivial weak solution of the System~\eqref{ellbvp-intro}. We also prove, via an elementary argument, a non-existence result. 
We provide some examples in order to illustrate the applicability of our theoretical results.

\section{Existence and non-existence results}
We make the following assumptions on the domain $\Omega$ and the operators $L_i$ and $B_i$ that occur in~\eqref{ellbvp-intro}
 (see \cite[Section~4 of Chapter~1]{Amann-rev} and \cite{lan1,lan2})):

\begin{enumerate}
\item $\Omega\subset \R^m$, $m\ge 2$, is a bounded domain such that its boundary $\partial \Omega$ is an $(m-1)$-dimensional $C^{2+\hat{\mu}}-$manifold for some $\hat{\mu}\in (0,1)$, such that $\Omega$ lies locally on one side of $\partial \Omega$ (see \cite[Section 6.2]{zeidler} for more details).
\item $L_i$ is a the second order elliptic operator given by
\begin{equation*}
L_i u(x)=-\sum_{j,l=1}^m a_{ijl}(x)\frac{\partial^2 u}{\partial x_j \partial x_l}(x)+\sum_{j=1}^m a_{ij}(x) \frac{\partial u}{\partial x_j} (x)+a_i (x)u(x), \quad \mbox{for $x\in \Omega$,}
\end{equation*}
where $a_{ijl},a_{ij},a_i\in C^{\hat{\mu}}(\overline{\Omega})$ for $j,l=1,2,\ldots,m$, $a_i(x)\ge 0$ on $\bar{\Omega}$, $a_{ijl}(x)=a_{ijl}(x)$ on $\bar{\Omega}$ for $j,l=1,2,\ldots,m$.  Moreover $L_i$ is strongly uniformly elliptic, that is, there exists $\bar{\mu}_{i0}>0$ such that 
$$\sum_{j,l=1}^m a_{ijl}(x)\xi_j \xi_l\ge \bar{\mu}_{i0} \|\xi\|^2 \quad \mbox{for $x\in \Omega$ and $\xi=(\xi_1,\xi_2,\ldots,\xi_m)\in\R^m$.}$$

\item $B_i$ is a boundary operator given by
$$B_i u(x)=b_i(x)u(x)+\delta_i \frac{\partial u}{\partial \nu}(x) \quad \mbox{for $x\in\partial \Omega$},$$
where $\nu$ is an outward pointing and nowhere tangent vector field on $\partial \Omega$ of class $C^{1+\hat{\mu}}$ (not necessarily a unit vector field),  $\frac{\partial u}{\partial \nu}$ is the directional derivative of $u$ with respect to $\nu$,  $b_i:\partial \Omega \to \R$ is of class $C^{1+\hat{\mu}}$ and moreover one of the following conditions holds:
\begin{enumerate}
\item $\delta_i =0$ and $b_i(x)\equiv 1$ (Dirichlet boundary operator).
\item $\delta_i =1$, $b_i(x)\equiv 0$ and $a_i(x)\not\equiv 0$ (Neumann boundary operator).
\item $\delta_i =1$, $b_i(x)\ge 0$ and $b_i(x)\not\equiv 0$ (Regular oblique derivative boundary operator).
\end{enumerate}
\end{enumerate}

It is known (see \cite{Amann-rev}, Section 4)  that, under the previous conditions, a strong maximum principle holds and, furthermore, 
given $g\in C^{\hat{\mu}}(\bar{\Omega})$, the boundary value problem 
 \begin{equation}
  \label{eqelliptic}
 \left\{
\begin{array}{ll}
   L_{i}u(x)=g(x), &  x\in \Omega, \\
 B_{i}u(x)=0, & x\in \partial \Omega,
\end{array}
\right.
\end{equation}
admits a unique classical solution $u\in C^{2+\hat{\mu}}(\bar{\Omega})$.

In order to seek solutions of the System~\eqref{ellbvp-intro}, we work in a suitable cone of positive functions. We recall that a \emph{cone} $P$ of a
real Banach space $X$ is a closed set with $P+P\subset P$,
  $\lambda P\subset P$ for all $\lambda\ge 0$ and  $P\cap(-P)=\{0\}$. A cone $P$ induces a partial ordering in
$X$ by means of the relation
$$ 
x\le y \quad \mbox{if and only if $y-x\in P$}.
$$

The cone $P$ is {\it normal} if there exists $d>0$ such that for all $x, y\in X$ with $0 \le x\le y$ then
$\|x\|\le d \|y\| $.
Note that every (closed) cone $P$ has the {\it Archimedean property}, that is, $n x\le y$ for all $n\in \N$ and some $y\in X$ implies $x\le 0$. In what follows, with abuse of notation, we will use the same symbol ``$\ge$" for the different cones appearing in the paper.

Now consider the (normal) cone of non-negative functions $P=C(\bar{\Omega},\R_+)$, then the solution operator $K_i:C^{\hat{\mu}}(\bar{\Omega})\to C^{2+\hat{\mu}}(\bar{\Omega})$ defined as $K_{i}g=u$ is linear, continuous and  (due to the maximum principle) positive, that is 
$K_{i}(P)\subset P$. It is known that $K$ can be extended uniquely to a continuous, linear and compact operator $K_{i}:C(\bar{\Omega})\to C(\bar{\Omega})$ (that we denote again by the same name). The following result (see \cite[Lemma 5.3]{amann-JFA}) provides further positivity properties of 
 the generalized solution operator.
 \begin{pro} \label{proK2} Let $e_i=K_{i}1\in C(\bar{\Omega},\R_+)\setminus \{0\}$. Then $K_{i}:C(\bar{\Omega})\to  C^1(\bar{\Omega})\subset C(\bar{\Omega})$ is $e$-positive (and in particular positive), that is for each $g\in C(\bar{\Omega},\R_+)\setminus \{0\}$ there exist $\alpha_g>0$ and $\beta_g>0$ such that $\alpha_g e_i \le K_{i} g \le \beta_g e_i$. 
 \end{pro}
Denote by $r(K_{i})$ the spectral radius of $K_{i}$. As a consequence of Proposition~\ref{proK2} and the Krein-Rutman theorem, it is known (for details see, for example, Lemma~3.3 of~\cite{lan2}) that  $r(K_{i})\in (0,+\infty)$ and there exists 
$\varphi_{i}\in P\setminus \{0\}$  such that 
\begin{equation}\label{igenfun}
\varphi_{i}=\mu_{i}K_{i}\varphi_{i},
\end{equation}{}
where $\mu_{i}=1/r(K_{i})$.

We utilize the space $C(\bar{\Omega},\R^n)$, endowed with the norm $\|u\|:=\displaystyle\max_{i=1,2,\ldots,n} \{\|u_i\|_{\infty}\}$, where $\|z\|_{\infty}=\displaystyle\max_{x\in \bar{\Omega}}|z(x)|$, and consider (with abuse of notation) the cone $P=C(\bar{\Omega},\R^n_+)$.

Given a nonempty set $D\subset C(\bar{\Omega},\R^n)$
we define 
$$D_I=\{u\in D: u(x)\in I \ \text{for all}\ x\in \bar{\Omega}\},$$
where $I=\prod_{i=1}^n I_i\subset \R^n,$ where each $I_i\subset \R$ is a closed nonempty interval.

Given a function $f_i:\bar{\Omega}\times I\to\R$ we define the Nemytskii (or superposition) operator $F_i$ in the following way
$$F_i(u)(x):=f_i(x,u(x)),\ \text{for}\ u\in C(\bar{\Omega}, I)\ \text{and}\ x\in \bar{\Omega}.$$

We now fix $I=\prod_{i=1}^n [0,\rho_i]$ and rewrite the elliptic System~\eqref{ellbvp-intro} as a fixed point problem in the product space of continuous functions by considering the operators $T,\Gamma:C(\bar{\Omega}, I) \to  C(\bar{\Omega},\R^n)$ given by 
\begin{align}
T(u):=(\lambda_i K_i F_i(u))_{i=1..n},\quad
\Gamma (u):=(\eta_i \gamma_i h_{i}[u] )_{i=1..n},
\end{align}
where 
$\gamma_i \in C^{2+\hat{\mu}}(\overline{\Omega})$ is the unique solution (nonnegative, due to the maximum principle, see \cite[Section~4 of Chapter~1]{Amann-rev}) of the BVP
 \begin{equation*}
 \left\{
\begin{array}{ll}
   L_{i} u(x)=0, &  x\in \Omega, \\
 B_{i} u(x)=1, & x\in \partial \Omega.
\end{array}
\right.
\end{equation*}

\begin{defn} We say that $u\in C(\bar{\Omega},I)$ is a {\it weak solution} of the System~\eqref{ellbvp-intro} if and only if $u$ is a fixed point of the operator $T+\Gamma$, that is, 
$$
u=Tu+\Gamma u=(\lambda_i K_i F_i(u)+\eta_i \gamma_i h_{i}[u] )_{i=1..n};
$$
if, furthermore, the components of $u$ are non-negative with $u_j\not\equiv 0$ for some $j$ we say that $u$ is a \emph{nonzero positive solution}.
\end{defn}

In the following Proposition we recall the main properties of the classical fixed point index for compact maps, for more details see~\cite{Amann-rev, guolak}. In what follows the closure and the boundary of subsets of a cone $\hat{P}$ are understood to be relative to $\hat{P}$.

\begin{pro}\label{propindex}
Let $X$ be a real Banach space and let $\hat{P}\subset X$ be a cone. Let $D$ be an open bounded set of $X$ with $0\in D\cap \hat{P}$ and
$\overline{D\cap \hat{P}}\ne \hat{P}$. 
Assume that $T:\overline{D\cap \hat{P}}\to \hat{P}$ is a compact operator such that
$x\neq Tx$ for $x\in \partial (D\cap \hat{P})$. Then the fixed point index
 $i_{\hat{P}}(T, D\cap \hat{P})$ has the following properties:
 
\begin{itemize}

\item[$(i)$] If there exists $e\in \hat{P}\setminus \{0\}$
such that $x\neq Tx+\lambda e$ for all $x\in \partial (D\cap \hat{P})$ and all
$\lambda>0$, then $i_{\hat{P}}(T, D\cap \hat{P})=0$.

\item[$(ii)$] If $Tx \neq \lambda x$ for all $x\in
\partial  (D\cap \hat{P})$ and all $\lambda > 1$, then $i_{\hat{P}}(T, D\cap \hat{P})=1$.

\item[$(iii)$] Let $D^{1}$ be open bounded in $X$ such that
$(\overline{D^{1}\cap \hat{P}})\subset (D\cap \hat{P})$. If $i_{\hat{P}}(T, D\cap \hat{P})=1$ and $i_{\hat{P}}(T,
D^{1}\cap \hat{P})=0$, then $T$ has a fixed point in $(D\cap \hat{P})\setminus
(\overline{D^{1}\cap \hat{P}})$. The same holds if 
$i_{\hat{P}}(T, D\cap \hat{P})=0$ and $i_{\hat{P}}(T, D^{1}\cap \hat{P})=1$.
\end{itemize}
\end{pro}

With these ingredients we can now state a result regarding the existence of positive solutions for the System~\eqref{ellbvp-intro}.
\begin{thm}\label{thmsol}Let  $I=\prod_{i=1}^n [0,\rho_i]$  and assume the following conditions hold. 

\begin{itemize}
\item[(a)] 
For every $i=1,2,\ldots,n$,  $f_i\in C(\bar{\Omega}\times I)$ and $f_i\geq 0$. Set
$$
M_i:=\max_{(x,u)\in\bar{\Omega}\times I } f_i (x,u).
$$
\item[(b)] There exist $\delta \in (0,+\infty)$, $i_0\in \{1,2,\ldots,n\}$ and $\rho_0 \in (0,\displaystyle\min_{i=1..n}{\rho_i})$ such that $$f_{i_0}(x,u)\ge \delta u_{i_0},\ \text{for every}\ (x,u)\in \bar{\Omega}\times I_{0},$$
where $I_{0}:=\prod_{i=1}^n [0,\rho_0].$

\item[(c)] For every $i=1,2,\ldots,n$, $h_{i}: P_I \to [0, +\infty )$ is continuous and 
$$H_i:=\sup_{u\in P_I }h_{i}[u]<+\infty.$$

\item[(d)]  For every $i=1,2,\ldots,n$ the following two inequalities are satisfied
\begin{equation}
\label{parrange}
\frac{\mu_{i_{0}}}{\delta}\leq \lambda_{i_0}\  \text{and}\ 
\lambda_i M_i  \| K_i(1)\|_{\infty} + \eta_i H_{i}\|\gamma_i\|_{\infty} \leq \rho_i.
\end{equation}
\end{itemize}

Then the System~\eqref{ellbvp-intro} has a nonzero positive weak solution $u$ such that $$\rho_0\leq \|u\|\ \text{and}\ \|u_i\|_{\infty}\leq \rho_i, 
\ \text{for every}\ i=1,2,\ldots,n.$$
\end{thm}

\begin{proof}
Take $P=C(\bar{\Omega},\R^n_+)$. Due to the assumptions above the operator $T+\Gamma$ maps $P_I$ into $P$ and is compact (the compactness of $T$ is well-known and $\Gamma$ is a finite rank operator). 
If $T+\Gamma$ has a fixed point either on $\partial {P_I}$ or $\partial {P_{I_{0}}}$ we are done. 

Assume now that $T+\Gamma$ is fixed point free on $\partial {P_I}\cup\partial {P_{I_{0}}}$, we are going to prove that $T+\Gamma$ has a fixed point in 
$ P_I\setminus (\partial {P_I}\cup P_{I_{0}})$.

We firstly prove, by means of (a), (c) and (d), that 
\begin{equation*} 
\sigma  u\neq Tu+\Gamma u\ \text{for every}\ u\in \partial P_{I}\
\text{and every}\  \sigma >1.
\end{equation*}
If this does not hold, then there exist $u\in \partial P_{I}$ and $\sigma >1$ such that $\sigma  u\neq Tu+\Gamma u$. Note that 
$\| u_j\|_{\infty} = \rho_j$ for some $j$ and
$\| u_i\|_{\infty} \leq \rho_i$ for every $i$. Furthermore
for every $x\in \overline{\Omega}$ we obtain
\begin{multline*}
\sigma u_{j}(x)=
\lambda_j K_jF_j(u)(x) + \eta_j h_{j} [u]\gamma_j(x)\leq \| \lambda_j K_jF_j(u) + \eta_j h_{j} [u]\gamma_{j}\|_{\infty}\\
\leq \|\lambda_j K_j(M_j) \|_{\infty}+ \|\eta_j H_{j}\gamma_{j}\|_{\infty} = \lambda_j M_j  \| K_j(1)\|_{\infty} + \eta_j H_{j}\|\gamma_j\|_{\infty} \leq \rho_j.
\end{multline*}
Taking the supremum over $\overline{\Omega}$ we obtain $\sigma  \rho_j\leq \rho_j$, a contradiction 
which yields $$i_{P}(T+\Gamma, P_I\setminus \partial {P_I} )=1.$$

We now consider $\varphi=(\varphi_1,\ldots,\varphi_n)$ where $\varphi_i$ is given by~\eqref{igenfun} and use (b) and (d) to show that 
\begin{equation*} 
u\neq Tu+\Gamma u +\sigma \varphi \ \text{for every}\ u\in \partial P_{I_0}\ \text{and every}\ \sigma  >0.
\end{equation*}
If not, there exists $u\in \partial P_{\rho_0}$ and $\sigma  >0$ such that
\begin{equation*} 
u= Tu+\Gamma u+\sigma  \varphi . 
\end{equation*}
Then we have $u\ge \sigma  \varphi$ and, in particular, $u_{i_0}\ge \sigma  \varphi_{i_0}$. For every $x\in \bar{\Omega}$ we have
\begin{multline*}
u_{i_0}(x)=(\lambda_{i_0}K_{i_0} F_{i_0} u)(x)+\eta_{i_0}  h_{i_0}[u]\gamma_{i_0}(x)   +\sigma  \varphi_{i_0}(x) \\
\ge (\lambda_{i_0}K_{i_0} \delta u_{i_0})(x) +\sigma  \varphi_{i_0}(x)\ge 
(\lambda_{i_0} \delta  K_{i_0}(\sigma  \varphi_{i_0}))(x) +\sigma  \varphi_{i_0}(x)\\
=  \frac{ \sigma  \lambda_{i_0} \delta}{\mu_{i_{0}}} \varphi_{i_0}  (x) +\sigma  \varphi_{i_0}(x)  \geq 2\sigma  \varphi_{i_0}(x).
\end{multline*}
By iterating the process, for $x\in \bar{\Omega}$, we get
$$
u_{i_0}(x)\ge n\sigma  \varphi_{i_0}(x) \ \text{for every}\ n\in\mathbb{N},
$$
 a contradiction, since $u$ is bounded. Thus we obtain $$i_{P}(T+\Gamma, P_{I_0} \setminus \partial {P_{I_{0}}})=0.$$
Therefore we have
$$i_{P}(T+\Gamma, P_I\setminus (\partial {P_I}\cup {P_{I_0}}))=i_{P}(T+\Gamma, P_I\setminus \partial {P_I} )-i_{P}(T+\Gamma, P_{I_0} \setminus \partial {P_{I_{0}}} )=1,$$
which proves the result.
\end{proof}

\begin{rem}
Note that, in the applications, sometimes it could be useful to replace the constants $M_i$ and $H_i$ with some majorants, say 
$\hat{M}_i$ and $\hat{H}_i$, at the cost of having to deal with the condition 
\begin{equation*}
\lambda_i \hat{M}_i  \| K_i(1)\|_{\infty} + \eta_i \hat{H}_{i}\|\gamma_i\|_{\infty} \leq \rho_i,\ \text{for every}\ i=1,2,\ldots, n,
\end{equation*}
more stringent than the corresponding
one occurring in~\eqref{parrange}.
\end{rem}
We now illustrate the applicability of Theorem~\ref{thmsol}.
\begin{ex}\label{exex}
Take $\Omega=\{x\in \mathbb{R}^2 : \|x\|_2<1\}$, and consider the system
\begin{equation} \label{example}
\left\{
\begin{array}{ll}
-\Delta u_1=\lambda_1 (|(u_1,u_2)|^\frac{1}{2} + \tan  |(u_1,u_2)| ),& \text{in }\Omega , \\
-\Delta u_2=\lambda_2 (1-\sin (u_2))|(u_1,u_2)|^2, & \text{in }\Omega , \\
u_1=\eta_1h_{1}[u],\ u_2=\eta_2h_{2}[u], & \text{on }\partial \Omega ,%
\end{array}%
\right. 
\end{equation}%
where $|(u_1,u_2)|=\max \{ |u_1|,|u_2|\}$,
$$h_{1}[u]=(u_1(0))^2+(u_2(0))^\frac{1}{2}\quad \mbox{and} \quad h_{2}[u]=(u_1(0))^\frac{1}{4}+\Bigl(\int_{\Omega}u_2(\xi)\,d\xi\Bigr)^2.$$

By direct calculation we obtain $K_1(1)=K_2(1)=\frac{1}{4}(1-x_1^2-x_2^2)$ and we may take $\gamma_1=\gamma_2\equiv 1$, this gives 
 $\|K_i(1)\|_{\infty}=\frac{1}{4}$ and $\|\gamma_i\|_{\infty}=1$ for $i=1,2$.

Fix $\rho_1,\rho_2=\frac{15 }{64}\pi$ and set
$$f_1(u_1,u_2)=|(u_1,u_2)|^\frac{1}{2} + \tan  |(u_1,u_2)|  \quad \mbox{and} \quad f_2(u_1,u_2)=(1-\sin (u_2)) |(u_1,u_2)|^2.$$

First of all note that given $\delta>0$, $f_1$ satisfies condition (b) in Theorem~\ref{thmsol} for $\rho_0$ sufficiently small, due to the behaviour near the origin. 

In the reminder of this example the numbers are rounded from above to the third decimal place unless exact. 

We have $M_1=f_1(\frac{15 }{64}\pi,\frac{15 }{64}\pi)\approx 1.765$ and $M_2=f_2(\frac{15 }{64}\pi, 0) =(\frac{15 }{64}\pi)^2 \approx 0.543$. Moreover, we can use the estimates $H_{1}\leq (\frac{15 }{64}\pi)^2+(\frac{15 }{64}\pi)^\frac{1}{2}\approx 1.401$ and $H_{2}\leq (\frac{15 }{64}\pi)^\frac{1}{4}+(\frac{15 }{64}\pi^2)^2\approx 6.278$.

By Theorem~\ref{thmsol}, the System~\eqref{example} has a nonzero positive solution $(u_1,u_2)$
such that $0<\|(u_1,u_2)\|\le \frac{15 }{64}\pi$ for every $\lambda_1,\lambda_2, \eta_1, \eta_2>0$ with
$$1.765\times\frac{\lambda_1}{4}+ 1.401\times \eta_1 \leq \frac{15 }{64}\pi\quad \text{and}\quad 0.543\times\frac{\lambda_2}{4}+ 6.278\times \eta_2 \leq \frac{15 }{64}\pi.$$
\end{ex}
We now prove, via an elementary argument, a non-existence result.
\begin{thm}\label{nonexthm}
Let $I=\prod_{i=1}^n [0,\rho_i]$ and 
assume that
for every $i=1,2,\ldots,n$ we have:
\begin{itemize}
\item
$f_i\in C(\bar{\Omega}\times I)$ and there exist $\tau_i \in (0,+\infty)$ such that
$$
0\leq f_i (x,u)\leq \tau_i u_i,\ \text{for every}\ (x,u)\in\bar{\Omega}\times I,
$$
\item
$h_{i}: P_I \to [0, +\infty )$ is continuous and there exist $\xi_i \in (0,+\infty)$
and
$$h_i[u]\leq \xi_i \|u\|, \ \text{for every}\ u \in P_I,$$
\item  the following inequality holds
\begin{equation}\label{nonexineq}
  \lambda_i \tau_i   \| K_i(1)\|_{\infty} + \eta_i \xi_i  \|\gamma_i\|_{\infty}<1.
 \end{equation}
\end{itemize}
Then the System~\eqref{ellbvp-intro} has at most the zero solution in $P_{I}$. 
\end{thm}
\begin{proof}
Assume, on the contrary, that there exists $u\in P_{I}$, $\|u\|= \sigma>0$, such that $u=Tu+\Gamma u$. Then there exists $j$ such that  $\|u_{j}\|_{\infty}=\sigma$. For $x\in \bar{\Omega}$ we have
\begin{multline*}
u_{j}(x)=
\lambda_j K_jF_j(u)(x) + \eta_j h_{j} [u]\gamma_j(x)\leq \| \lambda_j K_jF_j(u) + \eta_j h_{j} [u]\gamma_j\|_{\infty}\\
\leq \|\lambda_j K_j(\tau_j \sigma) \|_{\infty}+ \|\eta_j \xi_j \sigma\gamma_j\|_{\infty} = ( \lambda_j \tau_j   \| K_j(1)\|_{\infty} + \eta_j \xi_j  \|\gamma_j\|_{\infty}) \sigma < \sigma.
\end{multline*}
By taking the supremum over $\overline{\Omega}$, we obtain $\sigma  < \sigma$, a contradiction.
\end{proof}
We conclude by illustrating in the next example the applicability of Theorem~\ref{nonexthm}.
\begin{ex}\label{nonex}
Take $\Omega=\{x\in \mathbb{R}^2 : \|x\|_2<1\}$ and
consider the system
\begin{equation} \label{sysnonex}
\left\{
\begin{array}{ll}
-\Delta u_1=\lambda_1 u_{1}^2\sin(u_2) , & \text{in }\Omega , \\
-\Delta u_2=\lambda_2 u_{2}^4\cos(u_1), & \text{in }\Omega , \\
u_1=\eta_1h_{1}[u],\ u_2=\eta_2h_{2}[u], & \text{on }\partial \Omega ,%
\end{array}%
\right. 
\end{equation}%
where $h_{1}[u]=u_1(0)+(u_2(0))^2$ and $h_{2}[u]=u_1(0)+(u_2(0))^3$. 
First of all note that the trivial solution satisfies the System~\eqref{sysnonex}. Let us fix $I=[0, \frac{\pi}{4}]\times [0, \frac{\pi}{2}]$ and
note that for every $(x, u_1, u_2)\in \overline{\Omega}\times [0, \frac{\pi}{4}]\times [0, \frac{\pi}{2}]$ we have
$$
0\leq u_{1}^2\sin(u_2)\leq \frac{\pi}{4}u_{1},\ 0\leq u_{2}^4\cos(u_1)\leq \frac{\pi^3}{8} u_{2}.
$$
Furthermore for, $u\in P_{I}$, we have
$$
0\leq h_{1}[u]\leq \bigl( \frac{\pi}{2}+1 \bigr)\|u\|, \ 0\leq h_{2}[u]\leq \bigl( \frac{\pi^2}{4}+1 \bigr)\|u\|.
$$
Thus, in this case, the condition~\eqref{nonexineq} reads

\begin{equation}\label{nonexineqex}
\frac{\pi}{4}  \lambda_1   + \bigl( \frac{\pi}{2}+1 \bigr) \eta_1  <1\quad \text{and} \quad  \frac{\pi^3}{8}\lambda_2   +  \bigl( \frac{\pi^2}{4}+1 \bigr) \eta_2 <1.
\end{equation}
Therefore if~\eqref{nonexineqex} is satisfied, by Theorem~\ref{nonexthm} the System~\eqref{sysnonex} admits only the trivial solution in~$P_{I}$.
\end{ex}

\section*{Acknowledgement}
G. Infante was partially supported by G.N.A.M.P.A. - INdAM (Italy).

\end{document}